\documentclass[article]{IEEEtran}

\IEEEoverridecommandlockouts

\usepackage{hyperref}
\usepackage{url}
\usepackage{bbold}
\usepackage{amsfonts}
\usepackage{amsmath,amssymb}

\usepackage{amsthm}
\usepackage{graphicx}
\usepackage{hyperref}
\usepackage{wrapfig}
\usepackage{mathrsfs} 
\usepackage{algorithm,algorithmic}
\usepackage{array}
\usepackage{times}
\usepackage{url}
\usepackage{cite}
\usepackage{upgreek}
\usepackage{bbm}
\usepackage{float}
\usepackage{color}

\usepackage{pgfplots}
\usepackage{color}
\usepackage{flushend}
\usepackage{multirow}
\usepackage{rotating}
\usepackage{tikz,pgfplots}
\usepackage{standalone}

\newcommand\numberthis{\addtocounter{equation}{1}\tag{\theequation}}

\newtheorem{theorem}{Theorem}

\newtheorem{assumption}{Assumption}

\newtheorem{corollary}[theorem]{Corollary}

\newtheorem{lemma}[theorem]{Lemma}

\input{mysymbol.sty}

\begin{document}

\title{Distributed Online Optimization in Dynamic Environments Using Mirror Descent}

%\title{Decentralized Online Mirror Descent for Consensus Optimization in Dynamic Environments}

%\title{Mirror Descent for Decentralized Online Optimization in Dynamic Environments}

\author{Shahin Shahrampour and Ali Jadbabaie %Alexander Rakhlin, Vahid Tarokh, and 
\thanks{The authors gratefully acknowledge the support of ONR BRC Program on Decentralized, Online Optimization.}
\thanks{Shahin Shahrampour %and Vahid Tarokh are 
is with the Department of Electrical Engineering at Harvard University, Cambridge, MA 02138 USA. (e-mail: {\tt shahin@seas.harvard.edu}).}
%\thanks{Alexander Rakhlin is with the Department of Statistics at the University of Pennsylvania, Philadelphia, PA 19104 USA. (e-mail: {\tt rakhlin@wharton.upenn.edu}).}
\thanks{Ali Jadbabaie is with the Institute for Data, Systems, and Society at Massachusetts Institute of Technology, Cambridge, MA 02139 USA. (email: {\tt jadbabai@mit.edu}).}
}

\maketitle

% REQUIRED      %%%%%%%%%%%%%    ABSTRACT  %%%%%%%%%%%%%%%%%
\begin{abstract}
This work addresses decentralized online optimization in non-stationary environments. A network of agents aim to track the minimizer of a global time-varying convex function. The minimizer evolves according to a known dynamics corrupted by an unknown, unstructured noise. At each time, the global function can be cast as a sum of a finite number of local functions, each of which is assigned to one agent in the network. Moreover, the local functions become available to agents sequentially, and agents do not have a prior knowledge of the future cost functions. Therefore, agents must communicate with each other to build an online approximation of the global function. 
%The dynamics of the global minimizer is common knowledge; however, it can be corrupted by an unstructured noise, not allowing us to employ Kalman filters or particle filters for this framework. Hence, 
We propose a decentralized variation of the celebrated Mirror Descent, developed by Nemirovksi and Yudin. Using the notion of Bregman divergence in lieu of Euclidean distance for projection, Mirror Descent has been shown to be a powerful tool in large-scale optimization. Our algorithm builds on Mirror Descent, while ensuring that agents perform a consensus step to follow the global function and take into account the dynamics of the global minimizer. To measure the performance of the proposed online algorithm, we compare it to its offline counterpart, where the global functions are available a priori. The gap between the two is called dynamic regret. We establish a regret bound that scales inversely in the spectral gap of the network, and more notably it represents the deviation of minimizer sequence with respect to the given dynamics. We then show that our results subsume a number of results in distributed optimization. 
%We next consider stochastic optimization, where agents observe only noisy versions of their local gradients, and we prove that in this case, our regret bound holds true in the expectation sense. 
We demonstrate the application of our method to decentralized tracking of dynamic parameters and verify the results via numerical experiments. 
\end{abstract}

%%%%%%%%%%%%%%%%%%%%%%%%%%%%%%%%%%%%%%%%%%%%%%%

%%%%%%%%%%%%%%%%%%%%%%%    SECTIONS   %%%%%%%%%%%%%%%%

% !TEX root =  paper.tex

\section{Introduction}
Distributed convex optimization has received a great deal of interest in science and engineering. Classical engineering problems such as decentralized tracking, estimation, and detection are optimization problems in essence\cite{li2002detection,rabbat2004distributed,xiao2007distributed,lesser2012distributed,shahinjournal,nedic2015fast,qipeng2015distributed}, and early studies on parallel and distributed computation dates back to three decades ago with seminal works of  \cite{tsitsiklis1984problems,tsitsiklis1984distributed,bertsekas1989parallel}. In any decentralized scheme, the objective is to perform a {\it global} task, assigned to a number of agents in a network. Each individual agent has limited resources or partial information about the task. As a result, agents engage in {\it local} interactions to complement their insufficient knowledge and accomplish the global task. The use of decentralized techniques has increased rapidly since they impose low computational burden on agents and are robust to node failures as opposed to centralized algorithms which heavily rely on a single information processing unit. 

In distributed optimization, the main task is often minimization of a global convex function, written as the sum of local convex functions, where each agent holds a private copy of one specific local function. Then, based on a communication protocol, agents exchange local gradients to minimize the global cost function.

Decentralized optimization is a mature discipline in addressing problems dealing with {\it time-invariant} cost functions\cite{nedic2009distributed,johansson2009randomized,ram2010distributed,duchi2012dual,zhu2012distributed,jakovetic2014fast,shi2015extra}. However, in many real-world applications, cost functions vary over time. Consider, for instance, the problem of tracking a moving target, where the goal is to follow the position, velocity, and acceleration of the target. One should tackle the problem by minimizing a loss function defined with respect to these parameters; however, since they are {\it time-variant}, the cost function becomes {\it dynamic}. 

When the problem framework is dynamic in nature, there are two key challenges one needs to consider:
\begin{itemize}
\item[1)] Agents often observe the local cost functions in an {\it online} or {\it sequential} fashion, i.e., the local functions are revealed only after they make their instantaneous decision at each round, and they are unaware of future cost functions. In the last ten years, this problem (in the centralized domain) has been the main focus of the online optimization field in the machine learning community \cite{shalev2011online}.  
\item[2)] Any online algorithm should mimic the performance of its offline counterpart, and the gap between the two is called {\it regret}. The most stringent benchmark is an offline problem that aims to track the minimizer of the global cost function over time, which brings forward the notion of {\it dynamic} regret \cite{zinkevich2003online}. It is well-known that this benchmark makes the problem intractable in the worst-case. However, as studied in the centralized online optimization \cite{zinkevich2003online,hall2015online,besbes2015non,jadbabaie2015online}, the hardness of the problem can be characterized via a complexly measure that captures the variation in the minimizer sequence.
%A potential benchmark (offline problem) is minimization of the temporal average of the (global) cost functions, which is often called {\it static} regret. 
\end{itemize}

In this paper, we aim to address the above directions simultaneously. We consider an online optimization problem, where the global cost is realized sequentially, and the objective is to track the minimizer of the function. The dynamics of the minimizer is common knowledge, but the time-varying  minimizer sequence can deviate from this dynamics due to an {\it unstructured} noise. At each time step, the global function can be cast as sum of local functions, each of which is associated to one agent. Therefore, agents need to exchange information to solve the global problem.

Our multi-agent tracking setup is reminiscent of a distributed Kalman \cite{olfati2007distributed}. However, there are fundamental distinctions in our approach: {\it (i)} We do not assume that the minimizer sequence is corrupted with a Gaussian noise. Nor do we assume that this noise has a statistical distribution. Instead, we consider an adversarial-noise model with {\it unknown} structure. {\it (ii)} Agents observations are not necessarily linear; in fact, the observations are local gradients being potentially non-linear. Furthermore, our focus is on the {\it finite-horizon} analysis rather than asymptotic results. 

We also note that our setup differs from the distributed particle filtering\cite{gu2007distributed} as it is {\it online}, and agents receive only one observation per iteration. Moreover, we reiterate that the noise does not have a certain statistical distribution.

For this setup, we propose a decentralized version of the well-known Mirror Descent\footnote{Algorithms relying on Gradient Descent minimize Euclidean distance in the projection step. Mirror Descent generalizes the projection step using the concept of Bregman divergence \cite{yudin1983problem,beck2003mirror}. Euclidean distance is a special Bregman divergence that reduces Mirror Descent to Gradient Descent. Kullback-Leibler divergence is another well-known type of Bregman divergence (see e.g.   \cite{bauschke2001joint} for more details on Bregman divergence).}, developed by Nemirovksi and Yudin \cite{yudin1983problem}. Using the notion of Bregman divergence in lieu of Euclidean distance for projection, Mirror Descent has been shown to be a powerful tool in large-scale optimization. Our algorithm consists of three interleaved updates: {\it (i)} each agent follows the local gradient while staying close to previous estimates in the local neighborhood; {\it (ii)} agents take into account the dynamics of the minimizer sequence; {\it (iii)} agents average their estimates in their local neighborhood in a consensus step.

 Motivated by centralized online optimization, we use the notion of {\it dynamic} regret to characterize the difference between our online decentralized algorithm and its offline centralized version. We establish a regret bound that scales inversely in the spectral gap of the network, and more notably it represents the deviation of minimizer sequence with respect to the given dynamics. That is, it highlights the impact of the arbitrary noise driving the dynamical model of the minimizer. We further consider stochastic optimization, where agents observe only noisy versions of their local gradients, and we prove that in this case, our regret bound holds true in the expectation sense. 
 
Our main theoretical contribution is providing a comprehensive analysis on networked online optimization in dynamic setting. Our results subsume two important classes of decentralized optimization in the literature: {\it (i)} decentralized optimization of {\it time-invariant} objectives, and {\it (ii)} decentralized optimization of time-variant objectives over {\it fixed} variables. This generalization is an artifact of allowing dynamics in both objective and variable.   

We finally show that our algorithm is applicable to decentralized tracking of dynamic parameters. In fact, we show that the problem can be posed as the minimization of the square loss using Euclidean distance as the Bregman divergence. We then empirically verify that the tracking quality depends on how well the parameter follows its given dynamics.

\subsection{Related Literature}
This work is related to two distinct bodies of literature: {\it (i)} decentralized optimization, and {\it (ii)} online optimization in dynamic environments. Our goal in this work is to bridge the two and provide a general framework for decentralized {\it online} optimization in {\it non-stationary} environments. Below, we provide an overview of the related works to both scenarios:

\noindent
{\bf Decentralized Optimization:} There are a host of results in the literature on decentralized optimization for time-invariant functions. The seminal work of \cite{nedic2009distributed} studies distributed subgradient methods over time-varying networks and provides convergence analysis. The effect of stochastic gradients is then considered in \cite{ram2010distributed}. Shi et al. \cite{shi2015extra} prove fast convergence rates for Lipschitz-differentiable objectives by adding a correction term to the decentralized gradient descent algorithm. Of particular relevance to this work is \cite{li2016distributed}, where decentralized mirror descent has been developed for when agents receive the gradients with a delay. More recently, the application of mirror descent to saddle point problems is studied in \cite{li2016distributed2}. Moreover, Rabbat in \cite{rabbat2015multi} proposes a decentralized mirror descent for stochastic composite optimization problems and provide guarantees for strongly convex regularizers. In \cite{raginsky2012continuous}, Raginsky and Bouvrie investigate distributed stochastic mirror descent in the continuous-time domain. On the other hand, Duchi et al. \cite{duchi2012dual} study dual averaging for distributed optimization, and provide a comprehensive analysis on the impact of network parameters on the problem. The extension of dual averaging to online distributed optimization is considered in \cite{hosseini2013online}. Mateos-N{\'u}nez and Cort{\'e}s \cite{mateos2014distributed} consider online optimization using subgradient descent of local functions, where the graph structure is time-varying. In \cite{nedic2015decentralized}, a decentralized variant of Nesterov's primal-dual algorithm is proposed for online optimization. Finally, in \cite{7353155}, distributed online optimization is studied for strongly convex objective functions over time-varying networks.

 \noindent
{\bf Online Optimization in Dynamic Environments:} In online optimization, the benchmark can be defined abstractly in terms of a time-varying sequence, a particular case of which is the minimizer  sequence of a time-varying cost function. Several versions of the problem have been studied in the literature of machine learning in the centralized case. In \cite{zinkevich2003online}, Zinkevich develops the celebrated online gradient descent and considers its extension to time-varying sequences. The authors of \cite{hall2015online} generalize this idea to study time-varying sequences following given dynamics. Besbes et al. \cite{besbes2015non} restrict their attention to minima sequence and introduce a complexity measure for the problem in terms of variation in cost functions. For the same problem, the authors of \cite{jadbabaie2015online} develop an adaptive algorithm whose regret bound is expressed in terms of the variation of both functions and minima sequence, while in \cite{mokhtari2016online} an improved rate is derived for strongly convex objectives. Moreover, online dynamic optimization with linear objectives is discussed in \cite{lee2015resisting}. Fazlyab et al.  \cite{mahyar} consider interior point methods and provide continuous-time analysis for the problem. Finally, Yang et al. \cite{yang2016tracking} provide optimal bounds for when the minimizer belongs to the feasible set.

\subsection{Organization}
The paper is organized as follows. The notation, problem formulation, assumptions, and algorithm are described in Section \ref{prelim}. In Section \ref{theory}, we provide our theoretical results characterizing the behavior of the dynamic regret. Section \ref{application} is dedicated to application of our method to decentralized tracking of dynamic parameters. Section \ref{conclusion} concludes, and the proofs are given in Section \ref{appendix} (Appendix).

% !TEX root =  paper.tex

\section{Problem Formulation and Algorithm}\label{prelim}
\noindent
{\bf Notation:} We use the following notation in the exposition of our results:
\begin{center}
  \begin{tabular}{| c || l | }
    \hline
     $\vphantom{\sum^N} [n]$ &  The set $\{1,2,...,n\}$ for any integer $n$ \\ \hline 
     $\vphantom{\sum^N} x^\top$ & Transpose of the vector $x$ \\ \hline
     $\vphantom{\sum^N} x(k)$ & The $k$-th element of vector $x$ \\ \hline
     $\vphantom{\sum^N} I_n$ &  Identity matrix of size $n$ \\ \hline
     $\vphantom{\sum^N} \Delta_d$ &  The $d$-dimensional probability simplex \\ \hline
     $\vphantom{\sum^N} \inn{\cdot, \cdot}$ &  Standard inner product operator \\ \hline
     $\vphantom{\sum^N} \e{\cdot}$ &  Expectation operator \\ \hline
     $\vphantom{\sum^N} \norm{\cdot}_p$ &  $p$-norm operator \\ \hline
     $\vphantom{\sum^N} \norm{\cdot}_*$ &  The dual norm of $\norm{\cdot}$ \\ \hline
     %$\vphantom{\sum^N} \mathbb{1}$ &  Vector of all ones \\ \hline
     %$\norm{\mu-\pi}_{\text{TV}}$ &  Total variation distance between $\mu,\pi \in \Delta_m$ \\ \hline
     %$\vphantom{\sum^N} D_{KL}(\mu \| \pi)$ &  KL-divergence of $\pi\in \Delta_d$ from $\mu \in \Delta_d$  \\ \hline
     $\vphantom{\sum^N} \lambda_i(W)$ &  The $i$-th largest eigenvalue of matrix $W$ \\ \hline
     $\vphantom{\sum^N} \sigma_i(W)$ &  The $i$-th largest singular of matrix $W$ \\ \hline
  \end{tabular}
\end{center}
Throughout the paper, all the vectors are in column format.
\subsection{Decentralized Optimization in Dynamic Environments}
In this work, we consider an optimization problem involving a {\it global} convex function. We let $\X$ be a convex set and represent the global function by $f_t: \X \rightarrow \real$ at time $t$. The global function is time-variant, and the goal is to track the minimizer of $f_t(\cdot)$, denoted by $x^\star_t$. We address a finite-time problem whose {\it offline} and {\it centralized} version can be posed as follows 
\begin{equation}\label{problem}
\begin{aligned}
& \underset{x_1,\ldots,x_T}{\text{minimize}}
& & \sum_{t=1}^T f_t(x_t) \\
& \text{subject to}
& & x_t \in \X, \; t \in [T].
\end{aligned}
\end{equation}
However, we want to solve the problem in an {\it online} and {\it decentralized} fashion. In particular, the global function at each time $t$ can be written as the sum of $n$ {\it local} functions as
\begin{align}\label{decomposition}
f_t(x):=\frac{1}{n}\sum_{i=1}^nf_{i,t}(x),
\end{align}
 where $f_{i,t}: \X \rightarrow \real$ is a local convex function on $\X$ for all $i \in [n]$. We have a network of $n$ agents facing two challenges in solving problem \eqref{problem}: {\it (i)} agent $j\in [n]$ receives private information only about $f_{j,t}(\cdot)$ and does not have access to the global function $f_t(\cdot)$, which is common to decentralized schemes; {\it (ii)} The functions are revealed to agents sequentially along the time horizon, i.e., at any time instance $s$, agent $j$ has observed $f_{j,t}(\cdot)$ for $t<s$, whereas the agent does not know $f_{j,t}(\cdot)$ for $s \leq t\leq T$, which is common to online settings. 

The agents can exchange information with one another, and their relationship is encoded via an undirected graph $\mathcal{G}=(\mathcal{V}, \mathcal{E})$, where $\mathcal{V}=[n]$ denotes the set of nodes (agents), and $\mathcal{E}$ is the set of edges (links between agents). Each agent $i$ assigns a positive weight $[W]_{ij}$ for the information received from agent $j \neq i$. Hence, the set of neighbors of agent $i$ is defined as $\mathcal{N}_i:=\{j : [W]_{ij}>0\}$.

Note that our framework subsumes two important classes of decentralized optimization in the literature:
\begin{itemize}
\item[1)] Existing methods often consider {\it time-invariant} objectives (see e.g. \cite{nedic2009distributed,duchi2012dual,li2016distributed}). This is simply the special case where $f_t(x)=f(x)$ and $x_t=x$ in \eqref{problem}.
\item[2)] Online algorithms deal with {\it time-varying} functions, but often the network's objective is to minimize the temporal average of $\{f_t(x)\}_{t=1}^T$ over a fixed variable $x$ (see e.g. \cite{hosseini2013online,mateos2014distributed}). This can be captured by our setup when $x_t=x$ in \eqref{problem}.
\end{itemize}
To exhibit the online nature of the problem, the latter class in above is usually reformulated by a popular performance metric called {\it regret}. Since in that setup $x_t=x$ for $t\in [T]$, denoting by $x^\star:=\argmin_{x \in \X}\sum_{t=1}^Tf_t(x)$, the solution to problem \eqref{problem} becomes $\sum_{t=1}^T f_t(x^\star)$. Then, the goal of online algorithm is to mimic its offline version by minimizing the regret defined as follows
\begin{align}\label{staticregret}
\textbf{\textit{Reg}}^s_T=\frac{1}{n}\sum_{i=1}^n \sum_{t=1}^T f_t(x_{i,t}) - \sum_{t=1}^T f_t(x^\star),
\end{align}
where $x_{i,t}$ is the estimate of agent $i$ for $x^\star$ at time $t$. Moreover, the superscript ``s" reiterates the fact that the benchmark is minimum of the sum $\sum_{t=1}^T f_t(x)$ over a {\it static} or {\it fixed} comparator variable $x$ that resides in the set $\X$. In this setup, a successful algorithm incurs a sub-linear regret, which asymptotically closes the gap between the online algorithm and the  offline algorithm (when normalized by $T$).   

On the contrary, the focal point of this paper is to study the scenario where {\it functions} and {\it comparator variables} evolve simultaneously, i.e., the variables $\{x_t\}_{t=1}^T$ are not constrained to be fixed in \eqref{problem}. Let $x^\star_t:=\argmin_{x \in \X}f_t(x)$ be the minimizer of the global function at time $t$. Then, the solution to problem \eqref{problem} is simply $\sum_{t=1}^T f_t(x^\star_t)$. Therefore, to capture the online nature of problem \eqref{problem}, we reformulate it using the notion of {\it dynamic} regret as 
\begin{align}\label{regret}
\textbf{\textit{Reg}}^d_T=\frac{1}{n}\sum_{i=1}^n \sum_{t=1}^T f_t(x_{i,t}) -  \sum_{t=1}^T f_t(x^\star_t),
\end{align}
where $x_{i,t}$ is the estimate of agent $i$ for $x_t^\star$ at time $t$. The goal is to minimize the dynamic regret measuring the gap between the online algorithm and its offline version. The superscript ``d" indicates that the benchmark is the sum of minima $\sum_{t=1}^T f_t(x^\star_t)$ characterized by dynamic variables $\{x^\star_t\}_{t=1}^T$ that lie in the set $\X$.

It is well-known that the more stringent benchmark in the dynamic setup makes the problem intractable in the worst-case, i.e., achieving a sub-linear regret could be impossible. However, as studied in the centralized online optimization \cite{jadbabaie2015online,hall2015online,besbes2015non}, we would like to characterize the hardness of the problem via a complexly measure that captures the pattern of the minimizer sequence $\{x^\star_t\}_{t=1}^T$. More specifically, assuming that a dynamics $A$ is a {\it common} knowledge in the network, and 
\begin{align}\label{dynamics}
x^\star_{t+1}=Ax^\star_{t}+v_t,
\end{align}
we want to prove a regret bound in terms of 
\begin{align}\label{CT}
C_T:=\sum_{t=1}^T\norm{x^\star_{t+1}-Ax^\star_{t}}=\sum_{t=1}^T\norm{v_t},
\end{align}
which represents the deviation of minimizer sequence with respect to dynamics $A$. Note that generalizing the results to the time-variant case is straightforward, i.e., when $A$ is replaced by $A_t$ in \eqref{dynamics}. 

The problem setup \eqref{problem} coupled with the dynamics given in \eqref{dynamics} is reminiscent of distributed Kalman filtering\cite{olfati2007distributed}. However, there are fundamental distinctions here: {\it (i)} The mismatch noise $v_t$ is neither Gaussian nor of known statistical distribution. It can be thought as an adversarial noise with {\it unknown} structure, which represents the deviation from the dynamics\footnote{In online learning, the focus is not on distribution of data. Instead, data is thought to be generated arbitrarily, and its effect is observed through the loss functions\cite{shalev2011online}.}. {\it (ii)} Agents observations are not necessarily linear; in fact, the observations are local gradients of $\{f_{i,t}(\cdot)\}_{t=1}^T$ and are non-linear when the objective is not quadratic. Furthermore, another implicit distinction in this work is our focus on {\it finite-time} analysis rather than asymptotic results. 

We note that our framework also differs from distributed particle filtering\cite{gu2007distributed} since agents receive only one observation per iteration, and the mismatch noise $v_t$ has no structure or distribution. 

Having that in mind, to solve the online consensus optimization \eqref{regret}, we propose to decentralize the Mirror Descent algorithm \cite{yudin1983problem} and to analyze it in a dynamic framework. The appealing feature of Mirror Descent is extension of the projection step using Bregman divergence in lieu of Euclidean distance, which makes the algorithm applicable to a wide range of problems. Before defining Bregman divergence and elaborating the algorithm, we start by stating a couple of standard assumptions in the context of decentralized optimization.

\begin{assumption}\label{A1}
For any $i\in [n]$, the function $f_{i,t}(\cdot)$ is Lipschitz continuous on $\X$ with a uniform constant $L$. That is, $$|f_{i,t}(x)-f_{i,t}(y)|\leq L\norm{x-y},$$ for any $x,y \in \X$. This further implies that the gradient of $f_{i,t}(\cdot)$ denoted by $\nabla f_{i,t}(\cdot)$ is uniformly bounded on $\X$ by the constant $L$, i.e., we have $\norm{\nabla f_{i,t}(\cdot)}_* \leq L$.\footnote{This relationship is standard, see e.g. Lemma 2.6. in \cite{shalev2011online} for more details.}
\end{assumption}

\begin{assumption}\label{A2}
The network is connected, i.e., there exists a path from any agent $i\in [n]$ to any agent $j\in [n]$. Also, the matrix $W$ is doubly stochastic\footnote{For the sake of simplicity, we assume that the topology is time-invariant, and $W$ is fixed. The extension of problem to time-varying topology is straightforward, as previously investigated in the literature (see e.g. \cite{nedic2009distributed,duchi2012dual,li2016distributed}).} with positive diagonal. That is,
$$
\sum_{i=1}^n[W]_{ij}=\sum_{j=1}^n[W]_{ij}=1.
$$
\end{assumption}
The connectivity constraint in Assumption \ref{A2} guarantees the information flow in the
network. It simply implies uniqueness of  $\lambda_1(W)=1$ and warrants that other eigenvalues of $W$ are strictly less than one in magnitude \cite{horn2012matrix}.
\subsection{Decentralized Online Mirror Descent}
The development of Mirror Descent relies on the Bregman divergence outlined in this section. Consider a convex set $\X$ in a Banach space $\B$, and let $\R : \B \rightarrow \real$ denote a 1-strongly convex function on $\X$ with respect to a norm $\norm{\cdot}$. That is,
$$
\R(x)\geq \R(y)-\inn{\nabla\R(y),x-y}+\frac{1}{2}\norm{x-y}^2.
$$
for any $x,y \in \X$. Then, the Bregman divergence $\dr(\cdot,\cdot)$ with respect to the function $\R(\cdot)$ is defined as follows:
$$
\dr(x,y):=\R(x)-\R(y)-\inn{x-y,\nabla\R(y)}.
$$ 
Combining the two identities above yields an important property of the Bregman divergence, and for any $x,y \in \X$ we get 
\begin{align}\label{bregcond}
\dr(x,y) \geq \frac{1}{2}\norm{x-y}^2,
\end{align}
due to the strong convexity of $\R(\cdot)$. Two famous examples of Bregman divergence are the Euclidean distance and the Kullback-Leibler (KL) divergence generated from $\R(x)=\frac{1}{2}\norm{x}^2_2$ and $\R(x)=\sum_{i=1}^dx(i)\log x(i)-x(i)$, respectively.

\begin{assumption}\label{A3}
Let $x$ and $\{y_i\}_{i=1}^n$ be vectors in $\real^d$. The Bregman divergence satisfies the separate convexity in the following sense 
$$\dr(x,\sum_{i=1}^n\alpha(i)y_i) \leq \sum_{i=1}^n\alpha(i)\dr(x,y_i),$$
where $\alpha\in\Delta_n$ is on the $n$-dimensional simplex.
\end{assumption}
The assumption is satisfied for commonly used cases of Bregman divergence. For instance, the Euclidean distance evidently respects the condition. The KL-divergence also satisfies the constraint,  and we refer the reader to Theorem 6.4. in \cite{bauschke2001joint} for the proof.
\begin{assumption}\label{A4}
The Bregman divergence satisfies a Lipschitz condition of the form
\begin{align*}
| \dr(x,z)-\dr(y,z) | \leq K\|x- y\|, 
\end{align*}
for all $x,y,z \in \X$.
\end{assumption}
When the function $\R$ is Lipschitz on $\X$, the Lipschitz condition on the Bregman divergence is automatically satisfied. Again, for the Euclidean distance the assumption evidently holds. In the particular case of KL divergence, the condition can be achieved via mixing a uniform distribution to avoid the boundary. More specifically, consider $\R(x)=\sum_{i=1}^dx(i)\log x(i)-x(i)$ for which $|\nabla\R(x)|=|\sum_{i=1}^d \log x(i)| \leq d\log T$ as long as $x\in \{\mu : \sum_{i=1}^d\mu(i)=1 ;   \mu(i) \geq \frac{1}{T} , \ \forall i \in [d]\}$. Therefore, in this case the constant $K$ is of $\mathcal{O}(\log T)$ (see e.g. \cite{jadbabaie2015online} for more comments on the assumption).

We are now ready to propose a three-step algorithm to solve the optimization problem formulated in terms of dynamic regret in \eqref{regret}. Let us define $\nabla_{i,t}:=\nabla f_{i,t}(x_{i,t})$ as the shorthand for the local gradients. Noticing the dynamic framework, we develop the decentralized online mirror descent via the following updates\footnote{The algorithm is initialized at $x_{i,t}=\mathbb{0}$ to avoid clutter in the analysis. In general, any initialization could work for the algorithm.}
\begin{subequations}
\begin{align}
\hx_{i,t+1}&=\argmin_{x\in \X}  \big\{ \eta_t\inn{x, \nabla_{i,t}} + \dr(x,y_{i,t}) \big\}, \label{xhupdate}\\
x_{i,t}&=A \hx_{i,t}, \ \ \ \  \text{and}  \ \ \ \ \  y_{i,t}=\sum_{j=1}^n [W]_{ij} x_{j,t} \label{xyupdate},
\end{align}
\end{subequations}
where $\{\eta_t\}_{t=1}^T$ is the step-size sequence, and $A\in \real^{d \times d}$ is the given dynamics in \eqref{dynamics} which is a common knowledge. Recall that $x_{i,t}\in \real^d$ represents the estimate of agent $i$ for the global minimizer $x^\star_t$ at time $t$. The step-size sequence is non-increasing and positive. Our proposed methodology can also be recognized as the decentralized variant of the Dynamic Mirror Descent algorithm in \cite{hall2015online} though we restrict our attention only to linear dynamics. 

The update \eqref{xhupdate} allows the algorithm to follow the private gradient while staying close to the previous estimates in the local neighborhood. This closeness is achieved in the sense of minimizing the Bregman divergence. On the other hand, the first update in \eqref{xyupdate} takes into account the potential dynamics that the minimizer sequence follow, and the second update in \eqref{xyupdate} is the {\it consensus} term averaging the estimates in the local neighborhood. 

\begin{assumption}\label{A5}
The mapping $A$ is assumed to be {\it non-expansive}. That is, the condition
\begin{align*}
\dr\big(Ax,Ay\big) \leq \dr\big(x,y\big),
\end{align*}
holds for all $x,y \in \X$, and $\norm{A} \leq 1$.
\end{assumption} 
The assumption postulates a natural constraint on the mapping $A$: it does not allow the effect of a poor prediction (at one step) to be amplified as the algorithm moves forward. %For instance, consider the following state space model  
%\begin{align*}
%\theta_t=A\theta_{t-1}+u_t,
%\end{align*}
%where the system is driven by $u_t \in \R^d$. Let us restrict our attention to Euclidean distance as the Bregman divergence. Applying the same $u_t$ to two different states $\theta_{t-1}$ and $\theta'_{t-1}$, we have
%\begin{align*}
%\| \theta_t-\theta'_t\|_2^2=\| A\theta_{t-1}-A\theta'_{t-1}\|_2^2 \leq \sigma^2_1(A)\| \theta_{t-1}-\theta'_{t-1}\|_2^2,
%\end{align*}
%which is non-expansive whenever none of the singular values of $A$ is outside of the unit circle. When $A$ is symmetric, this coincides with the {\it marginal stability} of the system dynamics, preventing the states from growing unboundedly. Therefore, the assumption prohibits any wild oscillation of the minimizer sequence $\{x^\star_t\}_{t=1}^T$ in view of \eqref{dynamics}.

% !TEX root =  paper.tex

\section{Theoretical Results}\label{theory}
In this section, we state our theoretical results and their consequences. The proofs are presented later in the Appendix (Section \ref{appendix}). Our main result (Theorem \ref{theorem1}) proves a bound on the dynamic regret, which captures the deviation of the minimizer trajectory from the dynamics $A$ (tracking error) as well as the decentralization cost (network error). After stating the theorem, we show that our result recovers previous rates on decentralized optimization (static regret) once the tracking error is removed. Also, it recovers previous rates on centralized online optimization in dynamic setting when the network error is factored out. Therefore, we establish that our generalization is bona fide.      

\subsection{Preliminary Results}
We start with a convergence result on the local estimates, which presents an upper bound on the deviation of the local estimates at each iteration from their consensual value. A similar result has been proven in \cite{li2016distributed} for {\it time-invariant} functions {\it without} dynamics; however, the following lemma extends that of \cite{li2016distributed} to {\it online} setting and takes into account the {\it dynamics} $A$ in \eqref{xyupdate}.

\begin{lemma}\label{meandeviation}(Network Error)
Let $\X$ be a convex set in a Banach space $\B$, $\R : \B \rightarrow \real$ denote a 1-strongly convex function on $\X$ with respect to a norm $\|\cdot\|$, and $\dr(\cdot,\cdot)$ represent the Bregman divergence with respect to $\R$, respectively. Furthermore, assume that the local functions are Lipschitz continuous (Assumption \ref{A1}), the matrix $W$ is doubly stochastic (Assumption \ref{A2}), and the mapping $A$ is non-expansive (Assumption \ref{A5}). Then, the local estimates $\{x_{i,t}\}_{t=1}^T$ generated by the updates \eqref{xhupdate}-\eqref{xyupdate} satisfy 
 $$\norm{x_{i,t+1}-\bx_{t+1}}\leq L\sqrt{n}\sum_{\tau=0}^t\eta_\tau\sigma^{t-\tau}_2(W),$$
 for any $i\in [n]$, where $\bx_t:=\frac{1}{n}\sum_{i=1}^nx_{i,t}$.
\end{lemma}
It turns out that the error bound depends on the network parameter $\sigma_2(W)$ and the step-size sequence $\{\eta_t\}_{t=1}^T$. It is well-known that smaller $\sigma_2(W)$ results in closeness of estimates to their average by speeding up the mixing rate (see e.g. results of \cite{duchi2012dual}). For instance, when the communication is all-to-all, i.e., the graph is complete, $\sigma_2(W)=0$ and the mixing rate is most rapid since each agent receives the private gradients of others only after one iteration delay. On the other hand, a usual diminishing step-size, which asymptotically goes to zero, can guarantee asymptotic closeness; however, such step-size sequence is most suitable for {\it static} rather than {\it dynamic} environments. We will discuss the choice of step-size carefully when we state our main result. Before that, we need to state another lemma as follows.
\begin{lemma}\label{auxlemma}(Tracking Error)
 Let $\X$ be a convex set in a Banach space $\B$, $\R : \B \rightarrow \real$ denote a 1-strongly convex function on $\X$ with respect to a norm $\|\cdot\|$, and $\dr(\cdot,\cdot)$ represent the Bregman divergence with respect to $\R$, respectively. Furthermore, assume that the matrix $W$ is doubly stochastic (Assumption \ref{A2}), the Bregman divergence satisfies the Lipschitz condition and the separate convexity (Assumptions \ref{A3}-\ref{A4}), and the mapping $A$ is non-expansive (Assumption \ref{A5}). Then, it holds that
\begin{align*}
\frac{1}{n}\sum_{i=1}^n\sum_{t=1}^T\bigg(\frac{1}{\eta_t}&\dr(x^\star_t,y_{i,t})-\frac{1}{\eta_t}\dr(x^\star_t,\hx_{i,t+1})\bigg) \\
&\leq \frac{2R^2}{\eta_{T+1}}+\sum_{t=1}^T\frac{1}{\eta_{t+1}}\norm{x^\star_{t+1}-Ax^\star_{t}},
\end{align*}
where $R^2:=\sup_{x,y\in \X} \dr(x, y).$
\end{lemma}
In the update $\eqref{xhupdate}$, each agent $i$ calculates $\hx_{i,t+1}$, while staying close to $y_{i,t}$ by minimizing the Bregman divergence. Lemma \ref{auxlemma} establishes a bound on difference of these two quantities, when they are evaluated in the Bregman with respect to $x^\star_t$. The relation of left-hand side with dynamic regret is not immediate, and it becomes clear in the analysis. However, the term $\norm{x^\star_{t+1}-Ax^\star_{t}}=\norm{v_t}$ in the bound highlights the impact of mismatch noise $v_t$ in the tracking quality. Lemmata \ref{meandeviation} and \ref{auxlemma} disclose the critical parameters involved in the regret bound. We carefully discuss the consequences of these bounds in the subsequent section.

%Recall that our main objective is to control the dynamic regret \eqref{regret},  which suggests that a successful algorithm should be able to track the minimizer sequence $\{x^\star_t\}_{t=1}^T$ using local estimates $\{x_{i,t}\}_{t=1}^T$ for every $i\in [n]$. Though the regret definition is with respect to the global function $f_t(\cdot)$, 

\subsection{Finite-horizon Performance: Regret Bound}
We now state our main result on the {\it non-asymptotic} performance of the decentralized online mirror descent in dynamic environments. The succeeding theorem provides the regret bound in the general case, and it is followed by a corollary characterizing the regret rate for the optimized fixed step-size sequence. In particular, the theorem uses the results in the previous section to present an upper bound on the dynamic regret decomposed into tracking and network errors. 

\begin{theorem}\label{theorem1}
Let $\X$ be a convex set in a Banach space $\B$, $\R : \B \rightarrow \real$ denote a 1-strongly convex function on $\X$ with respect to a norm $\|\cdot\|$, and $\dr(\cdot,\cdot)$ represent the Bregman divergence with respect to $\R$, respectively. Furthermore, assume that the local functions are Lipschitz continuous (Assumption \ref{A1}), the matrix $W$ is doubly stochastic (Assumption \ref{A2}), the Bregman divergence satisfies the Lipschitz condition and the separate convexity (Assumptions \ref{A3}-\ref{A4}), and the mapping $A$ is non-expansive (Assumption \ref{A5}). Then, using the local estimates $\{x_{i,t}\}_{t=1}^T$ generated by the updates \eqref{xhupdate}-\eqref{xyupdate}, the regret \eqref{regret} can be bounded as 
\begin{align*}
\textbf{\textit{Reg}}^d_T &\leq {\tt E_{Track}}+{\tt E_{Net}},
\end{align*}
where
$$
{\tt E_{Track}}:=\frac{2R^2}{\eta_{T+1}}+\sum_{t=1}^T\frac{K}{\eta_{t+1}}\norm{x^\star_{t+1}-Ax^\star_{t}}+L^2\sum_{t=1}^T\frac{\eta_t}{2},
$$
and
$$
{\tt E_{Net}}:=4L^2\sqrt{n}\sum_{t=1}^T\sum_{\tau=0}^{t-1}\eta_\tau\sigma^{t-\tau-1}_2(W).
$$
\end{theorem}
\begin{corollary}\label{corollary1}
Under the same conditions stated in Theorem \ref{theorem1}, using the fixed step-size $\eta=\sqrt{(1-\sigma_2(W))C_T/T}$ yields a regret bound of order
$$
\textbf{\textit{Reg}}^d_T \leq \mathcal{O}\left(\sqrt{\frac{C_TT}{1-\sigma_2(W)}}\right),
$$
where $C_T=\sum_{t=1}^T\norm{x^\star_{t+1}-Ax^\star_{t}}$.
\end{corollary}
\begin{proof}
The proof of the corollary follows directly from substituting the step-size into the bound in Theorem \ref{theorem1}.
\end{proof}
Theorem \ref{theorem1} decomposes the upper bound into two terms for a general step-size $\{\eta_t\}_{t=1}^T$. In Corollary \ref{corollary1}, we fix the step-size and observe the role of $C_T$ in controlling the regret bound . As we recall from \eqref{dynamics}, this quantity collects mismatch errors $\{v_t\}_{t=1}^T$ that are not necessarily Gaussian or of some statistical distribution. In Section \ref{prelim}, we discussed that our setup generalizes some of previous works, and it is important to notice that our result recovers the corresponding rates when restricted to those special cases:
\begin{itemize}
\item[1)] When the global function $f_t(x)=f(x)$ is time-invariant, the minimizer sequence $\{x^\star_t\}_{t=1}^T$ is fixed, i.e., the mapping $A=I_d$ and $v_t=\mathbb{0}$ in \eqref{dynamics}. In this case in Theorem \ref{theorem1}, the term involving $\norm{x^\star_{t+1}-Ax^\star_{t}}$ in ${\tt E_{Track}}$ is equal to zero, and we can use the step-size sequence $\eta=\sqrt{(1-\sigma_2(W))/T}$ to recover the result of comparable algorithms, such as \cite{duchi2012dual} in which distributed dual averaging is proposed. 
\item[2)] The same argument holds when the global function is time-variant, but the comparator variables are fixed. In this case, the problem is reduced to minimizing the static regret \eqref{staticregret}. Since $\norm{x^\star_{t+1}-Ax^\star_{t}}=0$ again, our result recovers that of  \cite{hosseini2013online} on distributed online dual averaging. 
\item[3)] When the graph is complete, $\sigma_2(W)=0$ and the ${\tt E_{Net}}$ term in Theorem \ref{theorem1} vanishes. We then recover the results of  \cite{hall2015online} on centralized online learning in dynamic environments. 
\end{itemize}

As we mentioned earlier, when mismatch errors $\{v_t\}_{t=1}^T$ are large, the minimizer sequence $\{x^\star_t\}_{t=1}^T$ fluctuates drastically, and $C_T$ could become linear in time. The bound in the corollary is then not useful in the sense of keeping the dynamic regret sub-linear. Such behavior is natural since even in the centralized online optimization, the algorithm receives only a {\it single} gradient to predict the next step\footnote{Even in a more structured problem setting such as Kalman filtering, when we know the exact value of a state at a time step, we cannot exactly predict the next state, and we incur a minimum mean-squared error of  the size of noise variance.}. As discussed in Section \ref{prelim}, in this worst-case, the problem is generally intractable. However, our goal was to consider $C_T$ as a complexity measure of the problem environment and express the regret bound with respect to this parameter. In practice, if the algorithm is allowed to query {\it multiple} gradients per time, the error would be reduced, but this direction is beyond the scope of this paper.

\subsection{Optimization with Stochastic Gradients}
In many engineering applications such as decentralized tracking, learning, and estimation, agents observations are usually noisy. In this section, we demonstrate that the result of Theorem \ref{theorem1} does not rely on exact gradients, and it holds true in expectation sense when agents  follow stochastic gradients. Mathematically speaking, let $\F_t$ be the $\sigma$-field containing all information prior to the outset of round $t+1$. Let also ${\boldsymbol \nabla}_{i,t}$ represent the stochastic gradient observed by agent $i$ after calculating the estimate $\x_{i,t}$. Then, we define a stochastic oracle that provides noisy gradients respecting the following conditions 
\begin{align}\label{condition}
\e{\vphantom{\norm{{\boldsymbol \nabla}_{i,t}}_*^2}{\boldsymbol \nabla}_{i,t}\big\vert \F_{t-1}}=\nabla_{i,t} \ \ \ \ \ \ \ \ \ \ \e{\norm{{\boldsymbol \nabla}_{i,t}}_*^2\big\vert \F_{t-1}} \leq G^2.
\end{align}
The new updates take the following form
\begin{subequations}
\begin{align}
\xhx_{i,t+1}&=\argmin_{x\in \X}  \big\{ \eta_t\inn{x, {\boldsymbol \nabla}_{i,t}} + \dr(x,\y_{i,t}) \big\}, \label{xhupdate2}\\
\x_{i,t}&=A \xhx_{i,t}, \ \ \ \  \text{and}  \ \ \ \ \  \y_{i,t}=\sum_{j=1}^n [W]_{ij} \x_{j,t} \label{xyupdate2},
\end{align}
\end{subequations}
where the only distinction between \eqref{xhupdate2} and \eqref{xhupdate} is using the stochastic gradient in the former. A commonly used model to generate stochastic gradients satisfying \eqref{condition} is an additive zero-mean noise with bounded variance. We now discuss the impact of stochastic gradients in the following theorem.

\begin{theorem}\label{theorem2}
Let $\X$ be a convex set in a Banach space $\B$, $\R : \B \rightarrow \real$ denote a 1-strongly convex function on $\X$ with respect to a norm $\|\cdot\|$, and $\dr(\cdot,\cdot)$ represent the Bregman divergence with respect to $\R$, respectively. Furthermore, assume that the local functions are Lipschitz continuous (Assumption \ref{A1}), the matrix $W$ is doubly stochastic (Assumption \ref{A2}), the Bregman divergence satisfies the Lipschitz condition and the separate convexity (Assumptions \ref{A3}-\ref{A4}), and the mapping $A$ is non-expansive (Assumption \ref{A5}). Let the local estimates $\{\x_{i,t}\}_{t=1}^T$ be generated by updates \eqref{xhupdate2}-\eqref{xyupdate2}, where the stochastic gradients satisfy the condition \eqref{condition}. Then,
\begin{align*}
\e{\textbf{\textit{Reg}}^d_T} &\leq \frac{2R^2}{\eta_{T+1}}+\sum_{t=1}^T\frac{K}{\eta_{t+1}}\norm{x^\star_{t+1}-Ax^\star_{t}}\\
&+G^2\sum_{t=1}^T\frac{\eta_t}{2}
+4G^2\sqrt{n}\sum_{t=1}^T\sum_{\tau=0}^{t-1}\eta_\tau\sigma^{t-\tau-1}_2(W).
\end{align*}
\end{theorem}
The theorem indicates that when using stochastic gradients, the result of Theorem \ref{theorem1} holds true in expectation sense. Thus, the algorithm can be used in dynamic environments where agents observations are noisy.

% !TEX root =  paper.tex

\section{Numerical Experiment: State Estimation and Tracking Dynamic Parameters}\label{application}
The generality of Mirror Descent stems from the freedom over the selection of the Bregman divergence. A particularly well-known Bergman divergence is the Euclidean distance, which turns our framework to state estimation and tracking. In this section, we focus on this scenario as an application of our method. Distributed state estimation and tracking dynamic parameters has a long history in the literature of control and signal processing. However, there are key distinctions in our approach to the dynamical model of the parameter and agents observations. We elaborate on these differences as we describe our numerical experiment. 

Let us consider a slowly maneuvering target in the $2D$ plane and assume that each position component of the target evolves independently according to a near constant velocity model \cite{bar1987tracking}. The state of the target at each time consists of four components: horizontal position, vertical position, horizontal velocity, and vertical velocity. Therefore, representing such state at time $t$ by $x^\star_{t} \in \real^{4}$, the state space model takes the form
$$
x^\star_{t+1}=A x^\star_{t} + v_t,
$$ 
where $v_t \in \real^4$ is the system noise, and using $\otimes$ for Kronecker product, $A$ is described as
$$
A=I_2 \otimes \begin{bmatrix}
1 & \epsilon \\
0 & 1
\end{bmatrix},
$$
with $\epsilon$ being the sampling interval\footnote{The sampling interval of $\epsilon$ (seconds) is equivalent to the sampling rate of $1/\epsilon~(Hz)$.}. The goal is to track $x^\star_{t}$ using a network of agents. This problem has been studied in the context of distributed Kalman filtering \cite{olfati2007distributed,cattivelli2010diffusion}, state estimation \cite{khan2010connectivity,das2013distributed,han2015stochastic}, and particle filtering \cite{gu2007distributed,hlinka2012likelihood,li2015distributed}. However, as opposed to Kalman filtering, we need not assume that the system noise $v_t$ is Gaussian. Also, unlike particle filtering, we do not assume receiving a large number of samples (particles) per iteration since our setup is online, i.e., agents only observe one sample per iteration. Moreover, we do not assume a statistical distribution on $v_t$ in our analysis, which makes our framework different from state estimation. We have a model-free approach in which the noise can be deterministic with unknown structure, or even stochastic with dependence over time. For our experiment, we generate this noise according to a zero-mean Gaussian distribution with covariance matrix $\Sigma$ as follows
$$
\Sigma=\sigma^2_v I_2 \otimes \begin{bmatrix}
\epsilon^3/3 & \epsilon^2/2 \\
\epsilon^2/2 & \epsilon
\end{bmatrix}.
$$ 
We let the sampling interval be $\epsilon=0.1$ seconds which is equivalent to frequency $10~Hz$. The constant $\sigma^2_v$ is changed in different scenarios, so we describe the choice of this parameter later. Importantly, we remark that though this noise is generated randomly, it is fixed with each run of our experiment later. That is, the noise is generated once and remains fixed throughout, so it can be considered deterministic.

We consider a sensor network of $n=25$ agents located on a $5 \times 5$ grid. Agents aim to track the moving target $x^\star_{t}$ collaboratively. At time $t$, agent $i$ observes $\z_{i,t}$, a noisy version of one coordinate of $x^\star_{t}$ as follows
$$
\z_{i,t}=\ee_{k_i}^\top x^\star_{t} + \w_{i,t},
$$
where $\w_{i,t} \in \real$ denotes the observation noise, and $\ee_k$ is the $k$-th unit vector in the standard basis of $\real^4$ for $k\in \{1,2,3,4\}$. We divide agents into four groups, and for each group we choose one specific $k_i$ from the set $\{1,2,3,4\}$. Furthermore, the observation noise must satisfy the standard assumption of being zero-mean and finite-variance. Our results are not dependent on Gaussian noise, so we generate $\w_{i,t}$ independently from a uniform distribution on $[-1,1]$.

Though not locally observable to each agent, it is straightforward to see that the target $x^\star_{t}$ is {\it globally} identifiable from the standpoint of the whole network (see e.g. \cite{khan2010connectivity} for the exact definition of the global identifiability in a general tracking problem).

At time $t$, each agent $i$ forms an estimate $\x_{i,t}$ of $x^\star_{t}$ based on observations $\{\z_{i,\tau}\}_{\tau=1}^{t-1}$. After that, the new signal $\z_{i,t}$ becomes available to the agent. The online nature of the problem allows us to pose it as an instance of online optimization formulated in \eqref{regret}. To derive an explicit update for $\x_{i,t}$, we need to introduce the loss functions. We use the {\it local} square loss 
$$
f_{i,t}(x):=\e{\left(\z_{i,t}-\ee_{k_i}^\top x\right)^2  \big\vert  x^\star_{t}},
$$
for each agent $i$, resulting in the {\it network} loss  
\begin{align*}%\label{network loss}
f_t(x):=\frac{1}{n}\sum_{i=1}^n\e{\left(\z_{i,t}-\ee_{k_i}^\top x\right)^2 \big\vert x^\star_{t}}.
\end{align*}
In our experiment $v_t$ is a deterministic noise, but in both definitions $x^\star_{t}$ could be random in the case that $v_t$ is random, so we use the conditional expectation to be precise. Now using Euclidean distance as the Bregman divergence in updates \eqref{xhupdate2}-\eqref{xyupdate2}, we can derive the following update
\begin{align*}
\x_{i,t}&=\sum_{j=1}^n [W]_{ij} A\x_{j,t-1}+ \eta_t A\ee_{k_i} \left(\z_{i,t-1}-\ee_{k_i}^\top \x_{i,t-1}\right).
\end{align*}
We fix the step size to $\eta_t=\eta=0.5$ since using diminishing step size is not useful in tracking unless we have diminishing system noise \cite{acemoglu2008convergence}. The update is akin to {\it consensus+innovation} updates in the literature (see e.g. \cite{acemoglu2008convergence,shahrampour2013online,kar2012distributed}) though we recall that we did not analyze this update for a system noise $v_t$ with a statistical distribution.

It is proved in \cite{shahinACC2016} that in decentralized tracking, the dynamic regret can be presented in terms of the tracking error $\x_{i,t}-x^\star_{t}$ of all agents. More specifically, the dynamic regret averages the tracking error over space and time (when normalized by $T$). Exploiting this connection and combining that with the result of Theorem \ref{theorem2}, we observe that once the parameter does not deviate too much from the dynamics, i.e., when $\sum_{t=1}^T\norm{v_t}$ is small, the bound on the dynamic regret (or equivalently the collective tracking error) becomes small and vice versa. 

\begin{figure}[t!]
\centering
\includegraphics[trim = 10mm 53mm 0mm 58mm, clip, scale=0.45]{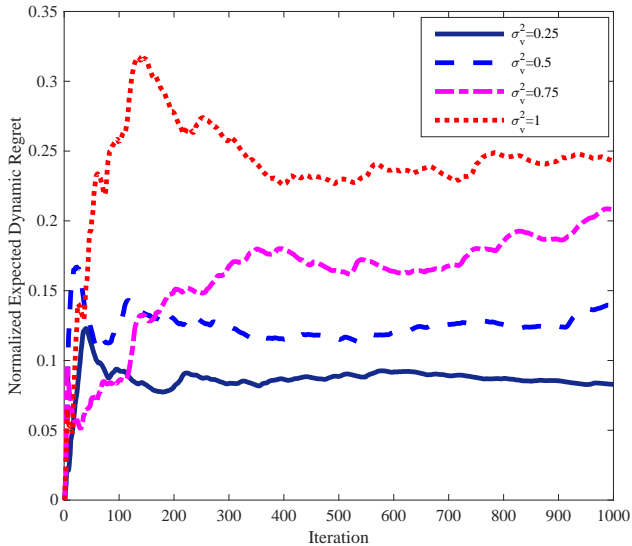}
\caption{The plot of dynamic regret versus iterations. Naturally, when $\sigma^2_v$ is smaller, the innovation noise added to the dynamics is smaller with high probability, and the network incurs a lower dynamic regret. In this plot, the dynamic regret is normalized by iterations, so the $y$-axis is $\e{\textbf{\textit{Reg}}^d_T}/T$.}
\label{Regretvs}
\end{figure}

\begin{figure}[h!]
\centering
\includegraphics[trim =8mm 47mm 0mm 47mm, clip, scale=0.44]{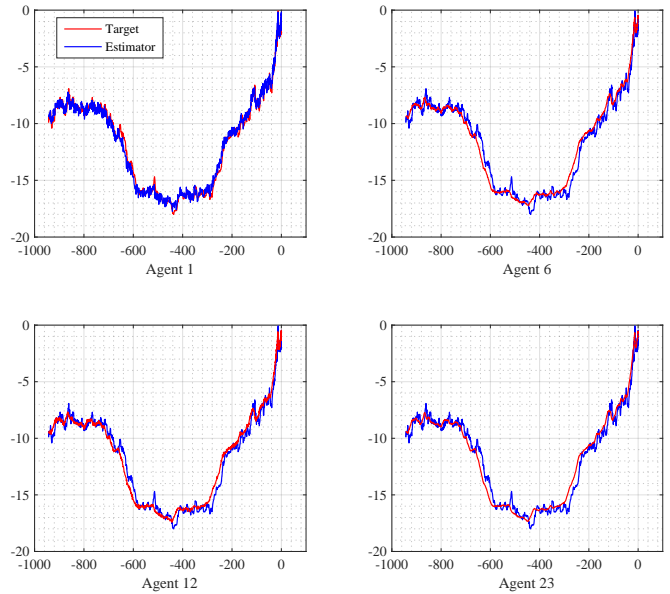}
\caption{The trajectory of $x^\star_t$ over $T=1000$ iterations is shown in red. We also depict the trajectory of the estimator $\x_{i,t}$ (shown in blue) for $i\in \{1,6,12,23\}$ and observe that it closely follows $x^\star_t$ in every case.}
\label{Agents}
\end{figure}

We demonstrate this intuitive idea by tuning $\sigma^2_v$. Larger values for $\sigma^2_v$ are more likely to cause deviations from the dynamics $A$; therefore, we expect large dynamic regret (worse performance) when $\sigma^2_v$ is large. In Fig. \ref{Regretvs}, we plot the dynamic regret for $\sigma^2_v\in \{0.25,0.5,0.75,1\}$. For each specific value of $\sigma^2_v$, we run the experiment 50 times and average out the dynamic regret over all runs. As we conjectured, the performance improves once $\sigma^2_v$ tends to smaller values.

Let us now focus on the case that $\sigma^2_v=0.5$. For one run of this case, we provide a snapshot of the target trajectory (in red) in Fig. \ref{Agents} and plot the estimator trajectory (in blue) for agents $i\in \{1,6,12,23\}$. While the dynamic regret can be controlled in the expectation sense (Theorem \ref{theorem2}), Fig. \ref{Agents} suggests that agents' estimators closely follow the trajectory of the moving target with high probability.

%%%%%%%%%%%%%%%%%%%%%%%%%%%%%%%%%%%%%%%%%%%%%%%

%%%%%%%%%%%%%%%%%%%%%%%   CONCLUSION     %%%%%%%%%%%%%%
\section{Conclusion}\label{conclusion}
The work unifies a number of frameworks in the literature by addressing {\it decentralized}, {\it online} optimization in {\it dynamic} environments. We considered tracking the minimizer of a global time-varying convex function via a network of agents. The minimizer of the global function has a dynamics known to agents, but an unknown, unstructured noise causes deviation from this dynamics. The global function can be written as a sum of local functions at each time step, and each agent can only observe its associated local function. However, these local functions appear sequentially, and agents do not have a prior knowledge of the future cost functions. 

Our proposed algorithm for this setup can be cast as a decentralized version of Mirror Descent. However, the algorithm possesses two additional steps to include agents interactions and dynamics of the minimizer. We used a notion of network dynamic regret to measure the performance of our algorithm versus its offline counterpart. We established that the regret bound scales inversely in the spectral gap of the network and captures the deviation of minimizer sequence with respect to the given dynamics. We next considered stochastic optimization, where agents observe only noisy versions of their local gradients, and we proved that in this case, our regret bound holds true in the expectation sense. We showed that our generalization is valid and convincing in the sense that the  results recover those of distributed optimization in online and offline setting. We also applied our method to decentralized tracking of dynamic parameters in the numerical experiments. 

Our work opens a few directions for future works. We conjecture that our theoretical results can be strengthened in a setup where agents receive multiple gradients per time step. However, as mentioned in Section \ref{theory}, this is still an open question. Also, the result of Corollary \ref{corollary1} assumes the step-size is tuned in advance. This would require the knowledge of $C_T$ or an upper bound on the quantity. For the centralized setting, one can potentially avoid the issue using doubling tricks which requires online accumulation of the mismatch noise $v_t$.  
However, it is more natural to consider that this noise is not fully observable in the decentralized setting. Therefore, an adaptive solution to step-size tuning remains open for the future investigation.

%%%%%%%%%%%%%%%%%%%%%%%%%%%%%%%%%%%%%%%%%%%%%%%

%%%%%%%%%%%%%%%%%%%%%%%%%%    APEENDIX    %%%%%%%%%%%%%
% !TEX root =  paper.tex

\section{Appendix}\label{appendix}
The following lemma is standard in the analysis of mirror descent. We state the lemma here and revoke it in our analysis later. 
\begin{lemma}[Beck and Teboulle \cite{beck2003mirror}]\label{beck}
Let $\X$ be a convex set in a Banach space $\B$, $\R : \B \rightarrow \real$ denote a 1-strongly convex function on $\X$ with respect to a norm $\norm{\cdot}$, and $\dr(\cdot,\cdot)$ represent the Bregman divergence with respect to $\R$, respectively. Then, any update of the form $$x^\star=\argmin_{x \in \X} \big\{ \inn{a, x}+\dr(x,c)\big\},$$ satisfies the following inequality $$\inn{x^\star-d, a } \leq \dr(d,c)-\dr(d,x^\star)-\dr(x^\star,c),$$ for any $d\in \X$. 
\end{lemma}

\subsection{Proof of Lemma \ref{meandeviation}}
\noindent
Applying Lemma \ref{beck} to the update \eqref{xhupdate}, we get
$$
\eta_t\inn{\hx_{i,t+1}-y_{i,t}, \nabla_{i,t}} \leq -\dr(y_{i,t}, \hx_{i,t+1})-\dr( \hx_{i,t+1},y_{i,t})
$$
In view of the strong convexity of $\R$, the Bregman divergence satisfies $\dr(x,y) \geq \frac{1}{2}\norm{x-y}^2$ for any $x,y \in \X$ (see \eqref{bregcond}). Therefore, we can simplify the equation above as follows
\begin{align}\label{eq1}
\eta_t\inn{y_{i,t}-\hx_{i,t+1}, \nabla_{i,t}} &\geq \dr(y_{i,t}, \hx_{i,t+1})+\dr( \hx_{i,t+1},y_{i,t})\notag\\
&\geq \norm{y_{i,t}- \hx_{i,t+1}}^2.
\end{align}
On the other hand, for any primal-dual norm pair it holds that 
\begin{align*}
\inn{y_{i,t}-\hx_{i,t+1}, \nabla_{i,t}}  &\leq  \norm{y_{i,t}- \hx_{i,t+1}}\norm{\nabla_{i,t}}_* \\
&\leq L\norm{y_{i,t}- \hx_{i,t+1}},
\end{align*}
using Assumption \ref{A1} in the last line. Combining above with \eqref{eq1}, we obtain 
\begin{align}\label{eq2}
\norm{y_{i,t}- \hx_{i,t+1}} \leq L\eta_t.
\end{align}
Letting $e_{i,t}:=\hx_{i,t+1}-y_{i,t}$, we can now rewrite update \eqref{xyupdate} as 
$$
\hx_{i,t+1}=\sum_{j=1}^n [W]_{ij} x_{j,t}+e_{i,t},
$$
which implies 
\begin{align}\label{eq3}
x_{i,t+1}=A\hx_{i,t+1}=\sum_{j=1}^n [W]_{ij} Ax_{j,t}+Ae_{i,t}.
\end{align}
Using Assumption \ref{A2} (doubly stochasticity of $W$), the above immediately yields 
\begin{align*}
\bx_{t+1}&:=\frac{1}{n}\sum_{i=1}^nx_{i,t+1}=\frac{1}{n}\sum_{i=1}^n\sum_{j=1}^n [W]_{ij} Ax_{j,t}+\frac{1}{n}\sum_{i=1}^nAe_{i,t}\\
&=\frac{1}{n}\sum_{j=1}^n \left(\sum_{i=1}^n[W]_{ij}\right) Ax_{j,t}+\frac{1}{n}\sum_{i=1}^nAe_{i,t}\\
\vphantom{\frac{1}{n}\sum_{j=1}^n \left(\sum_{i=1}^n[W]_{ij}\right) Ax_{j,t}+\frac{1}{n}\sum_{i=1}^nAe_{i,t}}&=A\bx_t+A\bar{e}_t,
\end{align*}
where $\bar{e}_t:=\frac{1}{n}\sum_{i=1}^ne_{i,t}$, and $\bx_t=\frac{1}{n}\sum_{i=1}^nx_{i,t}$ as defined in the statement of the lemma. As a result, 
\begin{align}\label{eq4}
\bx_{t+1}=\sum_{\tau=0}^tA^{t+1-\tau}\bar{e}_\tau.
\end{align}
On the other hand, stacking the local vectors $x_{i,t}$ and $e_{i,t}$ in \eqref{eq3} in the following form
\begin{align*}
x_t&:=[x_{1,t}^\top,x_{2,t}^\top,\ldots,x_{n,t}^\top]^\top\\
e_t&:=[e_{1,t}^\top,e_{2,t}^\top,\ldots,e_{n,t}^\top]^\top,
\end{align*}
and using $\otimes$ to denote the Kronecker product, we can write \eqref{eq3} in the matrix format as
\begin{align*}
\vphantom{\sum_{\tau=0}^t (W^{t-\tau}\otimes A^{t-\tau})(I_{n}\otimes A)e_{\tau}} x_{t+1}&=(W\otimes A)x_t+(I_{n}\otimes A)e_t\\
&=\sum_{\tau=0}^t (W\otimes A)^{t-\tau}(I_{n}\otimes A)e_{\tau}\\
&=\sum_{\tau=0}^t (W^{t-\tau}\otimes A^{t-\tau})(I_{n}\otimes A)e_{\tau}
\end{align*}
Therefore, using above, we have
$$
x_{i,t+1}=\sum_{\tau=0}^t\sum_{j=1}^n \left[W^{t-\tau}\right]_{ij}A^{t+1-\tau}e_{j,\tau}.
$$
Combining above with \eqref{eq4}, we derive
$$
x_{i,t+1}-\bx_{t+1}=\sum_{\tau=0}^t\sum_{j=1}^n\left(\left[W^{t-\tau}\right]_{ij}-\frac{1}{n}\right)A^{t+1-\tau}e_{j,\tau},
$$
which entails
\begin{align}\label{eq5}
\norm{x_{i,t+1}-\bx_{t+1}}\leq \sum_{\tau=0}^t\sum_{j=1}^n\left|\left[W^{t-\tau}\right]_{ij}-\frac{1}{n}\right|L\eta_\tau,
\end{align}
where we used $\norm{e_{i,\tau}}\leq L\eta_\tau$ obtained in \eqref{eq2} as well as the assumption $\norm{A}\leq 1$ (Assumption \ref{A5}). By standard properties of doubly stochastic matrices (see e.g. \cite{horn2012matrix}), the matrix W satisfies
$$
\sum_{j=1}^n\left|\left[W^{t}\right]_{ij}-\frac{1}{n}\right| \leq \sqrt{n}\sigma^t_2(W).
$$
Substituting above into \eqref{eq5} finishes the proof. \qed

\subsection{Proof of Lemma \ref{auxlemma}}
We start by adding, subtracting, and regrouping several terms as follows
\begin{align*}
\frac{1}{\eta_t}\dr(x^\star_t,&y_{i,t})-\frac{1}{\eta_t}\dr(x^\star_t,\hx_{i,t+1})=  \\
 +&\frac{1}{\eta_t}\dr(x^\star_t,y_{i,t})-\frac{1}{\eta_{t+1}}\dr(x^\star_{t+1},y_{i,t+1})\\
+&\frac{1}{\eta_{t+1}}\dr(x^\star_{t+1},y_{i,t+1})-\frac{1}{\eta_{t+1}}\dr(Ax^\star_{t},y_{i,t+1})\\
+&\frac{1}{\eta_{t+1}}\dr(Ax^\star_{t},y_{i,t+1})-\frac{1}{\eta_{t+1}}\dr(x^\star_{t},\hx_{i,t+1})\\
+&\frac{1}{\eta_{t+1}}\dr(x^\star_{t},\hx_{i,t+1})-\frac{1}{\eta_t}\dr(x^\star_{t},\hx_{i,t+1}) \numberthis\label{eq19}.
\end{align*}
We now need to bound each of the four terms above. For the second term, we note that
\begin{align}\label{eq10}
\dr(x^\star_{t+1},y_{i,t+1})-\dr(Ax^\star_{t},y_{i,t+1}) \leq K\norm{x^\star_{t+1}-Ax^\star_{t}},
\end{align}
by the Lipschitz condition on the Bregman divergence (Assumption \ref{A4}). Also, by the separate convexity of Bregman divergence (Assumption \ref{A3}) as well as stochasticity of $W$ (Assumption \ref{A2}), we have
\begin{align*}
\sum_{i=1}^n&\dr(Ax^\star_{t},y_{i,t+1})-\sum_{i=1}^n\dr(x^\star_{t},\hx_{i,t+1}) \\
&=\sum_{i=1}^n\dr(Ax^\star_{t},\sum_{j=1}^n[W]_{ij}x_{j,t+1})-\sum_{i=1}^n\dr(x^\star_{t},\hx_{i,t+1})\\
&\leq \sum_{i=1}^n\sum_{j=1}^n[W]_{ij}\dr(Ax^\star_{t},x_{j,t+1})-\sum_{i=1}^n\dr(x^\star_{t},\hx_{i,t+1})\\
&= \sum_{j=1}^n\dr(Ax^\star_{t},x_{j,t+1})\sum_{i=1}^n[W]_{ij}-\sum_{i=1}^n\dr(x^\star_{t},\hx_{i,t+1})\\
&= \sum_{j=1}^n\dr(Ax^\star_{t},x_{j,t+1})-\sum_{i=1}^n\dr(x^\star_{t},\hx_{i,t+1})\\
&= \sum_{i=1}^n\dr(Ax^\star_{t},A\hx_{i,t+1})-\sum_{i=1}^n\dr(x^\star_{t},\hx_{i,t+1})\leq 0, \numberthis \label{eq11}
\end{align*}
where the last inequality follows from the fact that $A$ is non-expansive (Assumption \ref{A5}). When summing \eqref{eq19} over $t\in [T]$ the first term telescopes, while the second and third terms are handled with the bounds in \eqref{eq10} and \eqref{eq11}, respectively. Recalling from the statement of the lemma that $R^2=\sup_{x,y\in \X} \dr(x, y)$, we obtain 
\begin{align*}
\frac{1}{n}\sum_{i=1}^n\sum_{t=1}^T\bigg(\frac{1}{\eta_t}&\dr(x^\star_t,y_{i,t})-\frac{1}{\eta_t}\dr(x^\star_t,\hx_{i,t+1})\bigg) \\
&\leq \frac{R^2}{\eta_1}+\sum_{t=1}^T\frac{K}{\eta_{t+1}}\norm{x^\star_{t+1}-Ax^\star_{t}}\\
&+R^2\sum_{t=1}^T\frac{1}{\eta_{t+1}}-\frac{1}{\eta_{t}}\\
&\leq \frac{2R^2}{\eta_{T+1}}+\sum_{t=1}^T\frac{K}{\eta_{t+1}}\norm{x^\star_{t+1}-Ax^\star_{t}},
\end{align*}
where we used the fact that the step-size is positive and decreasing in the last line.  \qed

\subsection{An Auxiliary Lemma}
In the proof of Theorem \ref{theorem1}, we make use of another technical lemma provided below.
\begin{lemma}\label{functiondeviation}
Let $\X$ be a convex set in a Banach space $\B$, $\R : \B \rightarrow \real$ denote a 1-strongly convex function on $\X$ with respect to a norm $\|\cdot\|$, and $\dr(\cdot,\cdot)$ represent the Bregman divergence with respect to $\R$, respectively. Furthermore, assume that the local functions are Lipschitz continuous (Assumption \ref{A1}), the matrix $W$ is doubly stochastic (Assumption \ref{A2}), the Bregman divergence satisfies the Lipschitz condition and the separate convexity (Assumptions \ref{A3}-\ref{A4}), and the mapping $A$ is non-expansive (Assumption \ref{A5}). Then, for the local estimates $\{x_{i,t}\}_{t=1}^T$ generated by the updates \eqref{xhupdate}-\eqref{xyupdate}, it holds that
\begin{align*}
\frac{1}{n}\sum_{i=1}^n\sum_{t=1}^T&\Big(f_{i,t}(x_{i,t})-f_{i,t}(x^\star_t)\Big)\leq \\
&\frac{2R^2}{\eta_{T+1}}+\sum_{t=1}^T\frac{K}{\eta_{t+1}}\norm{x^\star_{t+1}-Ax^\star_{t}}+L^2\sum_{t=1}^T\frac{\eta_t}{2}\\
&+2L^2\sqrt{n}\sum_{t=1}^T\sum_{\tau=0}^{t-1}\eta_\tau\sigma^{t-\tau-1}_2(W),
\end{align*}
where $R^2:=\sup_{x,y\in \X} \dr(x, y).$
\end{lemma}
\begin{proof}
In view of the convexity of $f_{i,t}(\cdot)$, we have
\begin{align*}
f_{i,t}(x_{i,t})-f_{i,t}(x^\star_t)& \leq \inn{\nabla_{i,t}, x_{i,t}-x^\star_t}\\
&=\inn{\nabla_{i,t}, \hx_{i,t+1}-x^\star_t}+\inn{\nabla_{i,t}, x_{i,t}-y_{i,t}}\\
&+\inn{\nabla_{i,t}, y_{i,t}-\hx_{i,t+1}} \numberthis \label{eq6}
\end{align*}
for any $i\in [n]$. We now need to bound each of the three terms on the right hand side of \eqref{eq6}. Starting with the last term and using boundedness of gradients (Assumption \ref{A1}), we have that
\begin{align*}
\vphantom{\frac{1}{2\eta_t}\norm{y_{i,t}-\hx_{i,t+1}}^2+\frac{\eta_t}{2}L^2}   \inn{\nabla_{i,t}, y_{i,t}-\hx_{i,t+1}} &\leq \norm{y_{i,t}-\hx_{i,t+1}}\norm{\nabla_{i,t}}_*\\
\vphantom{\frac{1}{2\eta_t}\norm{y_{i,t}-\hx_{i,t+1}}^2+\frac{\eta_t}{2}L^2} &\leq L\norm{y_{i,t}-\hx_{i,t+1}}\\
\vphantom{\frac{1}{2\eta_t}\norm{y_{i,t}-\hx_{i,t+1}}^2+\frac{\eta_t}{2}L^2} &\leq \frac{1}{2\eta_t}\norm{y_{i,t}-\hx_{i,t+1}}^2+\frac{\eta_t}{2}L^2, \numberthis \label{eq7}
\end{align*}
where the last line is due to AM-GM inequality. Next, we recall update \eqref{xyupdate} to bound the second term in \eqref{eq6} using Assumption \ref{A1} and \ref{A2} as 
\begin{align*}
\vphantom{2G^2\sqrt{n}\sum_{\tau=0}^t\sigma^{t-\tau}_2(W)\eta_\tau} &\inn{\nabla_{i,t}, x_{i,t}-y_{i,t}}=\inn{\nabla_{i,t}, x_{i,t}-\bx_{t} + \bx_{t}-y_{i,t}}\\
\vphantom{2G^2\sqrt{n}\sum_{\tau=0}^t\sigma^{t-\tau}_2(W)\eta_\tau} &~~~~~~~~~~~~~~=\inn{\nabla_{i,t}, x_{i,t}-\bx_{t}}+\sum_{j=1}^n[W]_{ij}\inn{\nabla_{i,t},\bx_{t}-x_{j,t}}\\
\vphantom{2G^2\sqrt{n}\sum_{\tau=0}^t\sigma^{t-\tau}_2(W)\eta_\tau} &~~~~~~~~~~~~~~\leq L\norm{x_{i,t}-\bx_{t}}+L\sum_{j=1}^n[W]_{ij}\norm{x_{j,t}-\bx_{t}}\\
\vphantom{2G^2\sqrt{n}\sum_{\tau=0}^t\sigma^{t-\tau}_2(W)\eta_\tau} &~~~~~~~~~~~~~~\leq 2L^2\sqrt{n}\sum_{\tau=0}^{t-1}\eta_\tau\sigma^{t-\tau-1}_2(W), \numberthis \label{eq8}
\end{align*}
where in the last line we appealed to Lemma \ref{meandeviation}. Finally, we apply Lemma \ref{beck} to \eqref{eq6} to get
\begin{align*}
\inn{\nabla_{i,t}, \hx_{i,t+1}-x^\star_t}& \leq \frac{1}{\eta_t}\dr(x^\star_t,y_{i,t})-\frac{1}{\eta_t}\dr(x^\star_t,\hx_{i,t+1})\\
&-\frac{1}{\eta_t}\dr(\hx_{i,t+1},y_{i,t})\\
& \leq \frac{1}{\eta_t}\dr(x^\star_t,y_{i,t})-\frac{1}{\eta_t}\dr(x^\star_t,\hx_{i,t+1})\\
&-\frac{1}{2\eta_t}\norm{\hx_{i,t+1}-y_{i,t}}^2, \numberthis \label{eq9}
\end{align*}
since the Bregman divergence satisfies $\dr(x,y) \geq \frac{1}{2}\norm{x-y}^2$ for any $x,y \in \X$. Substituting \eqref{eq7}, \eqref{eq8}, and \eqref{eq9} into the bound \eqref{eq6}, we derive 
\begin{align*}
f_{i,t}(x_{i,t})-f_{i,t}(x^\star_t)&\leq \frac{\eta_t}{2}L^2+2L^2\sqrt{n}\sum_{\tau=0}^{t-1}\eta_\tau\sigma^{t-\tau-1}_2(W)\\
&+\frac{1}{\eta_t}\dr(x^\star_t,y_{i,t})-\frac{1}{\eta_t}\dr(x^\star_t,\hx_{i,t+1}). \numberthis \label{eq14}
\end{align*}
Summing over $t\in [T]$ and $i\in [n]$, and applying Lemma \ref{auxlemma} completes the proof. \end{proof}

\subsection{Proof of Theorem \ref{theorem1}}
To bound the regret defined in \eqref{regret}, we start with
\begin{align*}
\vphantom{\frac{1}{n}\sum_{i=1}^n}f_t(x_{i,t}) - f_t(x^\star_t)&=f_t(x_{i,t})-f_t(\bx_t)+f_t(\bx_t) - f_t(x^\star_t)\\
\vphantom{\frac{1}{n}\sum_{i=1}^n}&\leq L\norm{x_{i,t}-\bx_t}+f_t(\bx_t)-f_t(x^\star_t)\\
\vphantom{\frac{1}{n}\sum_{i=1}^n}&= \frac{1}{n}\sum_{i=1}^nf_{i,t}(\bx_t)-\frac{1}{n}\sum_{i=1}^nf_{i,t}(x^\star_t)\\
\vphantom{\frac{1}{n}\sum_{i=1}^n}&+L\norm{x_{i,t}-\bx_t},
\end{align*}
where we used the Lipschitz continuity of $f_t(\cdot)$ (Assumption \ref{A1}) in the second line.
Using the Lipschitz continuity of $f_{i,t}(\cdot)$ for $i\in [n]$, we simplify above as follows  
\begin{align*}
f_t(x_{i,t}) - f_t(x^\star_t) &\leq \frac{1}{n}\sum_{i=1}^nf_{i,t}(x_{i,t})-\frac{1}{n}\sum_{i=1}^nf_{i,t}(x^\star_t)\\
\vphantom{\frac{1}{n}\sum_{i=1}^n}&+L\norm{x_{i,t}-\bx_t}+\frac{L}{n}\sum_{i=1}^n\norm{x_{i,t}-\bx_t}. \numberthis \label{eq11}
\end{align*}
Summing over $t\in [T]$ and $i\in [n]$, and applying Lemmata \ref{meandeviation} and \ref{functiondeviation} completes the proof. \qed

\subsection{Proof of Theorem \ref{theorem2}}
We need to rework the proof of Theorem \ref{theorem1} using stochastic gradients by tracking the changes. Following the lines in the proof of Lemma \ref{meandeviation}, equation \eqref{eq2} will be changed to 
$$
\norm{\y_{i,t}- \xhx_{i,t+1}} \leq \eta_t\norm{{\boldsymbol \nabla}_{i,t}}_*,
$$
yielding 
\begin{align}\label{eq12}
\norm{\x_{i,t+1}-\xbx_{t+1}}\leq \sqrt{n}\sum_{\tau=0}^t\eta_\tau\norm{{\boldsymbol \nabla}_{i,\tau}}_*\sigma^{t-\tau}_2(W).
\end{align}
On the other hand, at the beginning of Lemma \ref{functiondeviation}, we should use the stochastic gradient as
\begin{align*}
f_{i,t}(\x_{i,t})-&f_{i,t}(x^\star_t) \leq \inn{\nabla_{i,t}, \x_{i,t}-x^\star_t}\\
&= \inn{{\boldsymbol \nabla}_{i,t}, \x_{i,t}-x^\star_t}+\inn{\nabla_{i,t}-{\boldsymbol \nabla}_{i,t}, \x_{i,t}-x^\star_t}\\
&=\inn{{\boldsymbol \nabla}_{i,t}, \xhx_{i,t+1}-x^\star_t}+\inn{{\boldsymbol \nabla}_{i,t}, \x_{i,t}-\y_{i,t}}\\
&+\inn{{\boldsymbol \nabla}_{i,t}, \y_{i,t}-\xhx_{i,t+1}}+\inn{\nabla_{i,t}-{\boldsymbol \nabla}_{i,t}, \x_{i,t}-x^\star_t} 
\end{align*}
Moreover, as in Lemma \ref{meandeviation}, any bound involving $L$ which was originally an upper bound on the exact gradient must be replaced by the norm of stochastic gradient, which changes inequality \eqref{eq14} to 
\begin{align*}
&f_{i,t}(\x_{i,t})-f_{i,t}(x^\star_t)\leq \\
&~~~~~~~~~~~\frac{\eta_t}{2}\norm{{\boldsymbol \nabla}_{i,t}^2}_*+2\sqrt{n}\sum_{\tau=0}^{t-1}\eta_\tau\norm{{\boldsymbol \nabla}_{i,\tau}^2}_*\sigma^{t-\tau-1}_2(W)\\
&~~~~~~~~~~~+\frac{1}{\eta_t}\dr(x^\star_t,\y_{i,t})-\frac{1}{\eta_t}\dr(x^\star_t,\xhx_{i,t+1})\\
&~~~~~~~~~~~+\inn{\nabla_{i,t}-{\boldsymbol \nabla}_{i,t}, \x_{i,t}-x^\star_t}
\end{align*}
Then taking expectation from above, since
\begin{align*}
&\e{\inn{\nabla_{i,t}-{\boldsymbol \nabla}_{i,t}, \x_{i,t}-x^\star_t}}\\
&~~~~~~~~~~~=\e{\e{\inn{\nabla_{i,t}-{\boldsymbol \nabla}_{i,t}, \x_{i,t}-x^\star_t} \big\vert \F_{t-1}}}\\
&~~~~~~~~~~~=\e{\inn{\e{\nabla_{i,t}-{\boldsymbol \nabla}_{i,t}\big\vert \F_{t-1}}, \x_{i,t}-x^\star_t}}\\
&~~~~~~~~~~~=0, 
\end{align*}
using condition \eqref{condition}, we get 
\begin{align*}
&\e{f_{i,t}(\x_{i,t})}-f_{i,t}(x^\star_t)\leq\\
&~~~~~~~\frac{\eta_t}{2}\e{\norm{{\boldsymbol \nabla}_{i,t}}_*^2}+2\sqrt{n}\sum_{\tau=0}^{t-1}\eta_\tau\e{\norm{{\boldsymbol \nabla}_{i,\tau}}_*^2}\sigma^{t-\tau-1}_2(W)\\
&~~~~~~~+\e{\frac{1}{\eta_t}\dr(x^\star_t,\y_{i,t})-\frac{1}{\eta_t}\dr(x^\star_t,\xhx_{i,t+1})}
\end{align*}
Summing over $i \in [n]$ and $t\in [T]$, we apply bounded second moment condition \eqref{condition} and Lemma \ref{auxlemma} to get the same as bound as Lemma \ref{functiondeviation}, except for $L$ being replaced by $G$. Then the proof is finished once we return to \eqref{eq11}. \qed

\bibliographystyle{IEEEtran}
\bibliography{IEEEabrv,references}

\end{document}